\documentclass[a4paper,11pt]{amsproc}

    \usepackage[utf8]{inputenc}%
    \usepackage[english]{babel}%

    \usepackage{amsmath,amsfonts,amssymb,amsthm,amscd,stmaryrd}
    \usepackage[dvipsnames]{xcolor}
    \usepackage{tikz-cd}
    \usepackage{thmtools,thm-restate}
    \usepackage{setspace}
    \usepackage{a4wide}
    \usepackage{eucal}
    \usepackage{mathrsfs}
    \usepackage{mathtools}
    \usepackage{enumitem}
    \definecolor{citegreen}{rgb}{0,0.37,0.15}
    \definecolor{linkblue}{rgb}{0,0.4,1}
    \definecolor{citebordergreen}{rgb}{0,0.9,0.37}
    \usepackage[colorlinks=false, linkbordercolor={0 0 1}, linkcolor=blue, unicode=true, citecolor=citegreen, linkbordercolor={0 0.4 1}, citebordercolor={0 0.9 0.37}, pagebackref=true]{hyperref}
    
    \usepackage{xcolor, tikz}
    \usetikzlibrary{arrows}
    
      \theoremstyle{definition}
      \newtheorem{Def}{Definition}[section]%
      \newtheorem*{Def*}{Definition}%
      \theoremstyle{plain}
      \newtheorem{Prop}[Def]{Proposition}
      \newtheorem{Lemma}[Def]{Lemma}
      \newtheorem{Thm}[Def]{Theorem}

      \newtheorem{Conj}[Def]{Conjecture}
      \newtheorem*{Prop*}{Proposition}
      \newtheorem*{Thm*}{Theorem}%
      \newtheorem*{Cor*}{Corollary}%

      \theoremstyle{remark}
      \newtheorem{Rem}[Def]{Remark}%
      \newtheorem{Qu*}{Question}%

    \newcommand{\mc}{\mathcal}
    \newcommand{\mb}{\mathbb}
    \newcommand{\eps}{\varepsilon}

\DeclarePairedDelimiter\norm{\lVert}{\rVert}
    \DeclarePairedDelimiter\abs{\lvert}{\rvert}

    \DeclareMathOperator{\ev}{ev}

    \DeclareMathOperator{\PSL}{PSL}

    \DeclareMathOperator{\UU}{U}
    
    \DeclareMathOperator{\RUCB}{RUCB}

    \usetikzlibrary{calc,decorations}
    \pgfdeclaredecoration{arrows}{draw}{
    \state{draw}[width = \pgfdecoratedinputsegmentlength]{%
      \path [every arrow subpath/.try] \pgfextra{%
        \pgfpathmoveto{\pgfpointdecoratedinputsegmentfirst}%
        \pgfpathlineto{\pgfpointdecoratedinputsegmentlast}%
       };
    }}
    \tikzset{
                     base/.style = { circle, draw },
                   filled/.style = { base, fill = black!50 },
      every arrow subpath/.style = { ->, draw, thick }
    }

    \author{Vadim Alekseev}
    \address{V.~Alekseev, TU Dresden, 01062 Dresden, Germany}
    \email{vadim.alekseev@tu-dresden.de}
    \author{Max Schmidt}
    \address{M.~Schmidt, TU Dresden, 01062 Dresden, Germany}
    \email{max.schmidt1@tu-dresden.de}
    \author{Andreas Thom}
    \address{A.~Thom, TU Dresden, 01062 Dresden, Germany}
    \email{andreas.thom@tu-dresden.de}
    
    \onehalfspace
    
    \title{Amenability for unitary groups of $C^*$-algebras}
    \dedicatory{Dedicated to the memory of Eberhard Kirchberg}
    
\begin{document}
    \maketitle
    
    \begin{abstract} In this note we state a conjecture that characterizes unital $C^*$-algebras for which the unitary group is amenable as a topological group in the norm topology. We prove the conjecture for simple, separable, stably finite, unital, $\mathcal Z$-stable, UCT $C^*$-algebras with torsion-free $K_0$ using the progress on the Elliott classification program for nuclear $C^*$-algebras as well as Pestov's study of amenability of gauge groups. Based on work of Kirchberg, we provide a counterexample to a question of Ng, who proposed a different characterization in earlier work.
    \end{abstract}

    \section{Introduction}
    
    Amenability has been a central theme in the study of $C^*$-algebras and von Neumann algebras ever since the seminal work of Murray and von Neumann that got the entire subject started. Amenability of discrete groups has been important already in von Neumann's work on the Hausdorff paradox and has continued to be a key notion in geometric and combinatorial group theory. Amenability for topological groups beyond the realm of locally compact groups is a topic which is technically more subtle and various characterizations diverge in this setting. A topological group $G$ is called amenable if the algebra of bounded right-uniformly continuous functions admits a left invariant mean. Similarly, following Pestov \cite{MR4206535}, a topological group is called skew-amenable if the algebra of bounded right-uniformly continuous functions admits a right invariant mean. 
    
    The unitary group $\UU(A)$ of a $C^*$-algebra $A$ comes with two natural topologies, the weak topology and the norm topology. It is  natural to study the question for which $C^*$-algebras  the unitary group is amenable or skew-amenable in these topologies. Note that the norm topology is {\rm SIN}, i.e. there are small conjugation invariant neighborhoods of the neutral element, and thus amenability and skew-amenability of ${\rm U}(A)$ are equivalent with respect to this topology.
    
    A central notion in the study of $C^*$-algebras is nuclearity, a certain finite-dimensional approximation property, which resembles features of amenability in the context of operator algebras. 
    Various characterizations of nuclearity are known, including the work of Paterson~\cite{MR1076577}, showing that a $C^*$-algebra is nuclear if and only if its unitary group is amenable in the weak topology. To the best of our knowledge, no conclusive study has been carried out to characterize amenability of the unitary group of a $C^*$-algebra in the norm topology, however first results were obtained in the work of Ng \cite{MR2269417}.
    
    Seminal work of Haagerup and Connes \cite{MR0493383, MR723220} resulted in the equivalence of nuclearity of a $C^*$-algebra and amenability of it as a Banach algebra. However, a second and stronger notion comes up naturally: Ozawa \cite{MR3156986} studied {\it symmetric amenability} and proved that it is equivalent to nuclearity plus the {\it quotient tracial state property} (QTS property). A $C^*$-algebra is said to have the QTS property if all its non-trivial quotients admit a non-trivial trace.  We want to put forward the following conjecture:
    
    \begin{Conj} \label{conj1}
    The following conditions are equivalent for a unital, separable $C^*$-algebra $A$.
    \begin{enumerate}
        \item[$(i)$] $A$ is nuclear and has the QTS property,
        \item[$(ii)$] $A$ is symmetrically amenable,
        \item[$(iii)$] $A$ is strongly amenable,
        \item[$(iv)$] ${\rm U}(A)$ is amenable in the norm topology,
        \item[$(v)$] ${\rm U}(A)$ is skew-amenable in the weak topology.
    \end{enumerate}
    \end{Conj}
    
    As mentioned above, it has been proven by Ozawa \cite{MR3156986} that the first two items above are equivalent. The implication $(ii) \Rightarrow (iii)$ is a long-standing conjecture, see \cite{MR320764}. The implication from $(iv)$ to $(iii)$ will be proved in this article using results on F\o lner sets \cite{MR3809992}, see Theorem \ref{ivimpliesiii}. Finally, $(iv)\Rightarrow{}(v)$ follows since the norm topology is SIN (thus, as mentioned before, amenability and skew-amenability in norm are equivalent) and skew-amenability is preserved if one weakens the topology. Thus, the state of the art can be summarized as follows:
    $$(i) \Leftrightarrow (ii) \Leftarrow (iii) \Leftarrow (iv) \Rightarrow (v).$$
    
    Moreover, let us note that the conjecture holds if $A$ does not have the QTS property, as all five conditions are known to imply it.
    
    In general, it seems surprisingly hard to show that ${\rm U}(A)$ is amenable in the norm topology unless $A$ is finite-dimensional, in which case ${\rm U}(A)$ is compact and hence amenable. In this case, amenability is easily seen to be inherited by inductive limits, so $(iv)$ holds as a consequence for AF-algebras.
    
    Using ideas from stochastic analysis, Pestov showed that the unitary group of $M_n(C[0,1])$ is amenable in the norm topology, see \cite{MR4206535} and the references therein. The main result in the present article is a proof of Conjecture \ref{conj1} for a specific class of simple $C^*$-algebras that arises as inductive limits on one-dimensional noncommutative CW-complexes, using recent progress in the Elliott classification program and work of Pestov mentioned above. These algebras are natural examples of nuclear algebras with the QTS property.
    
    \begin{Thm} \label{main}
    Let $A$ be an inductive limit of one-dimensional noncommutative CW-complexes. Then, $\UU(A)$ is amenable in the norm topology. In particular, Conjecture \ref{conj1} holds for all nuclear, simple, separable, $\mathcal Z$-stable, unital $C^*$-algebras which satisfy the UCT and have torsion-free $K_0$.
    \end{Thm}
    
    The proof relies on recent progress in the Elliott classification program, see \cite{MR4228503}. We will comment in Section \ref{exrem} on further directions concerning applications of classification results in order to resolve Conjecture \ref{conj1} for particular classes of $C^*$-algebras.
    
    \section{Notions of amenability}
    
    \subsection{Strong amenability and symmetric amenability}
    
    Let $A$ be a complex algebra and consider $A \otimes_{\mathbb C} A$ with the natural bi-module structure over $A$. A Banach algebra $A$ is said to be amenable if there exists a net $\Delta_i \in A \otimes_{\mathbb C} A$ such that $\|\Delta_i\|_{\wedge}$ is uniformly bounded, $m(\Delta_i)$ is an approximate unit for $A$ and
    $\|a \Delta_i - \Delta_i a\|_{\wedge} $ tends to zero for each $a \in A.$ Here, $m \colon A \otimes_{\mathbb C} A \to A$ denotes the multiplication and $\|.\|_{\wedge}$ the projective tensor norm on $A \otimes_{\mathbb C}A.$ Such a net $(\Delta_i)_i$ is called approximate diagonal for $A$.
    
    Following \cite{MR1388200} we say that a Banach algebra is symmetrically amenable if $\Delta_i$ can be chosen invariant under the flip $x \otimes y \mapsto y \otimes x.$
    
    Ozawa \cite{MR3156986} proved the following theorem.
    
    \begin{Thm}[Ozawa]
    Let $A$ be a unital $C^*$-algebra. The following are equivalent:
    \begin{enumerate}
        \item $A$ is nuclear and has the QTS property.
        \item $A$ is symmetrically amenable,
        \item $A$ has a symmetric approximate diagonal in the set
        $$\left\{\sum_{i} x_i^* \otimes x_i \mid \sum_{i} \|x_i\|^2 \leq 1 \right\}.$$
    \end{enumerate}
    \end{Thm} 
    
    A possible different notion is the notion of strong amenability, where one requires that a symmetric approximate diagonal can be found in the convex hull of the set $\{u^* \otimes u \mid u \in {\rm U}(A)\}$. It is clear that strong amenability implies symmetric amenability. As mentioned before, it is an long-standing open problem if strong amenability and symmetric amenability are equivalent, see for example \cite{MR320764}.
    
    An important example of a strongly amenable $C^*$-algebra was found by Rosenberg in this study of groups crossed products. He showed that $(\mc O_n \otimes \mathcal K)^+$ is strongly amenable, whereas $\mc O_n$ itself is obviously not, see \cite{MR467331}.
    
    \subsection{F\o lner sets for topological groups}
    In this section we recall basic concepts of amenability for topological groups and particularly a generalisation of the F\o lner criterion for these groups which provides a useful tool for the proof of one implication in Conjecture \ref{conj1}. The following definitions can be found in \cite{MR2098458}.
    \begin{Def}
    Consider a topological group $G$.
    \begin{enumerate}
    \item A bounded complex valued function $f\colon G\to\mb C$ is said to be \emph{right uniformly continuous} if $\lim\limits_{x\to1}\norm{f_x-f}=0$ where $f_x(y):= f(xy)$ for all $y\in G$. We denote the space of all these functions by $\RUCB(G)$. Note, that $f_x\in\RUCB$ whenever $f\in\RUCB$. Furthermore, we set $f^x(y)\coloneqq f(yx)$ for all $y\in G$.
        \item Let $\mu \in\RUCB(G)^\ast$ be a functional. Then, $m$ is a
    \emph{left invariant mean} if $\mu(1)=\norm{\mu}=1$ and $\mu(f_x)=\mu(f)$ for all $x\in G$ and $f\in \RUCB(G)$, and a \emph{right invariant mean} if $\mu(1)=\norm{\mu}=1$ and $\mu(f^x)=\mu(f)$ for all $x\in G$ and $f\in \RUCB(G)$.
    \item The group $G$ is called \emph{amenable} if it has a left invariant mean and \emph{skew-amenable} if it has a right invariant mean.
    \end{enumerate} 
    \end{Def}
    
    The following characterization of amenability was obtained in \cite{MR3809992}. 
    
    \begin{Thm}
    \label{metrisableFolner}
    Let $G$ be a metrisable topological group with a right invariant metric $d$, that generates the topology. Then, the following are equivalent.
    \begin{enumerate}
    \item $G$ is amenable as a topological group.
    \item For every $\eps>0$, every finite set $E\subset G$ and every $\delta>0$ there exists a non-empty finite set $F\subset G$ such that for all $g\in E$
    there exists a bijection $\phi \colon F \to gF$ such that
    $$|\{h \in F \mid d(h,\phi(h))< \delta \}| \geq (1-\varepsilon)|F|.$$
    \end{enumerate}
    \end{Thm}We call a map $\phi$ as in the preceding theorem a $(1-\varepsilon,\delta)$-matching with respect to $E$. The sets $F$ arising as in the theorem are called F\o lner sets. It is standard to see that in case $G$ is SIN, one can choose the sets $F$ to be symmetric, i.e., we have $F=F^{-1}$.
    A more or less immediate application of this characterization is implication $(iv) \Rightarrow (iii)$ in Conjecture \ref{conj1} that we will now prove:
    
    \begin{Thm} \label{ivimpliesiii}
    Let $A$ be a unital $C^*$-algebra. If ${\rm U}(A)$ is amenable in the norm topology, then $A$ is strongly amenable.
    \end{Thm}
    \begin{proof}
    Since ${\rm U}(A)$ is amenable in the norm topology, Theorem \ref{metrisableFolner} leads to the existence of symmetric F\o lner set $F:=F_{(E,\varepsilon)}$ for a finite $E\subset {\rm U}(A)$ and $\varepsilon>0$. Let $g\in E$ and $\phi\colon F\to gF$ be a $(1-\eps,\eps)$-matching with respect to $E$. We define the net
    \begin{align*}
    \Delta_{F_{(E,\varepsilon)}}=\frac{1}{\abs{F}}\sum_{u\in F}u\otimes u^\ast
    \end{align*}
    where the index set is the set of pairs $(E,\varepsilon)$, where  $E\subset{\rm U}(A)$ is finite and $\varepsilon>0$.
    We now want to show that the net $(\Delta_{F_{(E,\varepsilon)}})_{(E,\varepsilon)}$ is an approximate diagonal. Since $F \subset {\rm U}(A)$, the element $\Delta_{F}$ is a contraction and $m(\Delta_{F_{(E,\varepsilon)}})=1$ for every $(E,\varepsilon)$. Let $F':= \{ h\in F \mid d(h,\phi(h))< \delta \}$ and $F'':= g^{-1}(\phi(F'))$. Using the matching $\phi$ as well as the unitarity of $u$ and $g$ we get
    \begin{align*}
    \norm*{\frac{1}{\abs{F}}\sum_{u\in F\setminus F''}gu\otimes u^\ast}_\wedge \leq \frac{1}{\abs{F}}\abs{F\setminus F''}\leq\eps.
    \end{align*}
    For convenience, we write $x=_\delta y$ if $\|x-y\|\leq \delta$. Hence,
    \begin{align*}
    g\Delta_{F_{(E,\varepsilon)}}&=\frac{1}{\abs{F}}\sum_{u\in F} gu\otimes u^\ast=\frac{1}{\abs{F}}\sum_{u\in F''}gu\otimes u^\ast+\frac{1}{\abs{F}}\sum_{u\in F\setminus F''}gu\otimes u^\ast\\
    &=_\eps \frac{1}{\abs{F}}\sum_{u\in F''}gu\otimes u^\ast=\frac{1}{\abs F}\sum_{\tilde u\in F'}\phi(\tilde u)\otimes \phi(\tilde u)^\ast g\\
    &=_{2\eps} \frac{1}{\abs{F}}\sum_{\tilde u\in F'}\tilde u\otimes \tilde u^\ast g=_\eps \frac{1}{\abs{F}}\sum_{\tilde u\in F}\tilde u\otimes \tilde u^\ast g\\
    &=\Delta_{F_{(E,\varepsilon)}}g
    \end{align*}
    since $\phi(\tilde u)=gu$ and $\tilde u=_\eps \phi(\tilde u)$.
    Altogether, $\norm{g\Delta_{F_{E}}-\Delta_{F_{E}}g}\leq 4\eps$ holds for an arbitrary $g\in E$. 
    It follows that $\|a \Delta_{F_E} - \Delta_{F_E}a\| \leq 16 \|a\| \varepsilon$ for any particular $a \in A$ if $E$ is large enough. Indeed,  every contraction in $A$ is a linear combination of four unitaries in $A$. Finally, note that $\Delta_{F}$ is invariant under the flip if $F$ is symmetric. This proves the claim.
    \end{proof}
    
    \subsection{Constructions with amenable groups}
    
    Let us summarize various constructions that preserve the class of amenable groups, see Proposition 4.1 in \cite{MR3644011}.
    
    \begin{Thm} \label{amen}
    Let G be a topological group.
    \begin{enumerate}
        \item \label{amenopen}
        If $G$ is amenable, then every open subgroup of $G$ is amenable.
        \item \label{amenunion}
        If $G$ is a directed union of a family $(H_\alpha)_{\alpha \in A}$ of closed subgroups, and if each $H_{\alpha}$ is amenable, then $G$ is amenable.
        \item \label{amenext} If $G$ has an amenable closed normal subgroup N such that the
    quotient $G/N$ is amenable, then $G$ is amenable.
        \item \label{amendense}
        Let $H$ be a topological group such that there exists a continuous
    homomorphism $H \to G$ with dense image; if $H$ is amenable,
    then so is $G$.
        \item If $H$ is a dense subgroup of $G$, then $G$ is amenable if and only
    if $H$, endowed with the induced topology, is amenable.
    \end{enumerate}
    \end{Thm}
    
    The following lemma was proved in \cite[Corollary 1.2]{PS23}.
    	
    	\begin{Lemma}[Pestov-Schneider]\label{co-compact}
    		Let $G$ be an amenable {\rm SIN} group and let $H \subset G$ be a closed co-compact subgroup. Then $H$ is amenable.
    	\end{Lemma}
    
    \section{Proof of the main theorem}
    \subsection{Noncommutative CW-complexes}
    
    Noncommutative CW-complexes were introduced in \cite{MR1716199} in arbitrary dimensions. We are interested in the one-dimensional case, which was studied in detail in \cite{MR2970466}.
    \begin{Def}
    \label{nccw}
    A $C^*$-algebra $A$ is said to be a \emph{one-dimensional noncommutative CW-complex}  if it is the pullback given by the following diagram:
    \begin{equation*}
    \begin{tikzcd}
    A & {C([0,1],F)}\\
    E & {F\oplus F}
    \arrow["{\pi_1}", from=1-1, to=1-2]
    \arrow["\ev_0\oplus\ev_1", from=1-2, to=2-2]
    \arrow["{\pi_2}", from=1-1, to=2-1]
    \arrow["\gamma", from=2-1, to=2-2]
    \end{tikzcd}
    \end{equation*}
    where $E$ and $F$ are finite dimensional $C^*$ -algebras, $\ev_0$ and $\ev_1$ the evaluation maps at the endpoints.
    \end{Def}
    
    The theorem below can be found in \cite[Corollary 6]{MR4206535} and is a corollary of \cite[Theorem 1.1]{MR1174157}.
    
    \begin{Thm}[Pestov] \label{pestov}
    The group $C([0,1],{\rm U}(n))$ is amenable in the norm topology.
    \end{Thm}
    
    The history of this result is quite remarkable. After a conjecture of Carey--Grundling \cite{MR2098458}, Pestov deduced Theorem \ref{pestov} by reduction to work on stochastic analysis of Malliavin--Malliavin \cite{MR1174157}; however, at the same time Pestov's work \cite[Sections 3, 4]{MR4206535} gives an self-contained proof of a stronger statement from which Theorem \ref{pestov} follows directly. We generalize Pestov's theorem as follows:
    
    \begin{Thm} \label{amenabilityU(A)nccw}
    Let $A$ be a one-dimensional noncommutative CW complex. Then ${\rm U}(A)$ is amenable in the norm topology.
    \end{Thm}
    \begin{proof}
    We use the notation of Definition \ref{nccw}. Setting $K=\ker\gamma$ and decomposing $E$ into $E=E'\oplus K$, we see that ${\rm U}(A)={\rm U}(A')\times {\rm U}(K)$ where $A'$ is the pullback along the map $\gamma|_{E'}$. Thus, without loss of generality we can assume $\gamma\colon E\to F\oplus F$ to be injective.
    
    In this latter case, considering the pullback diagram
    \begin{equation*}
    \begin{tikzcd}
    {{\rm U}(A)} & {C([0,1],{\rm U}(F))} \\
    {{\rm U}(E)} & {{\rm U}(F)\times{\rm U}(F)}
    \arrow["{\pi_1}", from=1-1, to=1-2]
    \arrow["{\pi_2}", from=1-1, to=2-1]
    \arrow["\gamma", from=2-1, to=2-2]
    \arrow["{\mathrm{ev}_0\times\mathrm{ev}_1}", from=1-2, to=2-2]
    \end{tikzcd}
    \end{equation*}
    and observing that ${\rm U}(E) \subset {\rm U}(F) \times {\rm U}(F)$ is co-compact, we conclude that ${\rm U}(A)$ is co-compact in $${\rm U}(C([0,1],F)) = \prod_{i=1}^k C([0,1],{\rm U}(n_i)),$$ where $F= \bigoplus_{i=1}^k M_{n_i}(\mathbb C)$. Now, amenability of ${\rm U}(A)$ follows from Lemma \ref{co-compact} and Theorem \ref{pestov}.
    \end{proof}
    
    \subsection{Consequences of the Elliott program}
    
    Relying on Elliott's seminal work providing concrete approximately sub-homogeneous $C^*$-algebras realizing a particular Elliott invariant \cite[Theorem 5.2.3.2]{MR1661611}, Thiel clarified in \cite{Thiel} that in case that $K_0(A)$ is torsion-free, one-dimensional noncommutative CW complexes are sufficient as building blocks -- see also the remarks in the lower part of \cite[page 87]{MR1661611}. 
    
    \begin{Thm}\label{thiel}
    Every nuclear, simple, separable, stably finite, unital, $\mathcal Z$-stable, UCT $C^*$-algebra $A$ such that $K_0(A)$ is torsion-free is an inductive limit of one-dimensional noncommutative CW-complexes.
    \end{Thm}
    
    This result appeared before it was proved in \cite{MR4228503} that separable, unital, simple, nuclear, $\mathcal{Z}$-stable, UCT $C^*$-algebras satisfy the Elliott conjecture; it did not mention the UCT. The above variant follows directly by combining \cite[Corollary D]{MR4228503} with the argument in \cite[Corollary 7.11]{Thiel}, using an easy observation that one-dimensional noncommutative CW-complexes and their limits satisfy the UCT (see \cite[Proposition 2.3]{MR894590}). It is hard to attribute Theorem \ref{thiel} to anyone, as it emerged from various sources -- we certainly claim no originality other than observing it.
    
    \begin{proof}[Proof of Theorem \ref{main}]
    To deduce the first statement from Theorem \ref{amenabilityU(A)nccw}, we need to pass to a direct limit which requires a bit of care. Let $A = \varinjlim A_n$ with canonical homomorphisms $\psi_n\colon A_n\to A$, then $\bigcup_n \UU(\psi_n(A_n))$ is a dense subgroup of $\UU(A)$. Thus, in view of Theorem \ref{amen} (\ref{amenunion})
    it suffices to show that each of the groups $\UU(\psi_n(A_n))$ is amenable. Note that one-dimensional noncommutative CW-complexes have stable rank one. Since $\psi_n(A_n)$ has also stable rank one, its unitary group can be written as an extension
    \[
    1\to \UU_0 (\psi_n(A_n)) \to \UU(\psi_n(A_n)) \to K_1(\psi_n(A_n))\to 1,
    \]
    see \cite[Theorem 2.9]{MR887221}. Since $\UU(A_n)$ is amenable by Theorem \ref{amenabilityU(A)nccw}, so is its open subgroup $\UU_0(A_n)$ using Theorem \ref{amen} (\ref{amenopen}).
    Furthermore, $\UU_0 (\psi_n(A_n)) = \psi_n(\UU_0 (A_n))$ is a quotient of $\UU_0(A_n)$ and therefore also amenable, see Theorem \ref{amen} (\ref{amendense}).
    Altogether, we deduce that $\UU(\psi_n(A_n))$, being an extension of an abelian group by an amenable group, is itself amenable by Theorem \ref{amen} (\ref{amenext}).
    This finishes the proof of the first statement. The second statement follows by applying Theorem \ref{thiel}.
    \end{proof}
    \section{Examples and remarks} \label{exrem}
    
    \subsection{Further directions}
    
    There have been prior attempts to characterize the class of $C^*$-algebras whose unitary group which is amenable in the norm topology, see \cite{MR2269417}. However, the following theorem gives a negative answer to Question $1.1$ from \cite{MR2269417}.
    
    \begin{Thm}
    There exists a stably finite, nuclear, unital $C^*$-algebra, such that ${\rm U}(A)$ is not amenable in the norm topology.
    \end{Thm}
    \begin{proof}
    Kirchberg \cite[Corollary 1.4 (v)]{MR1322641} provides a unital subalgebra $A$ of the CAR-algebra such that $\mc O_{\infty}$ is a quotient of $A$ by an AF-ideal. Therefore $A$ is an extension of a nuclear algebra by a nuclear ideal and therefore nuclear. Moreover, it clearly admits a faithful tracial state and is thus stably finite, however ${\rm U}(A)$ cannot be amenable in the norm topology since $A$ does not have the QTS property.
    \end{proof}
    
    \begin{Rem}
    The group ${\rm U}(A)$ in the above theorem is another example of a non-amenable closed subgroup of an amenable {\rm SIN} group (in this case the unitary group of the CAR-algebra), a phenomenon first discovered in \cite{MR3795479}. In view of Lemma \ref{co-compact}, this shows that ${\rm U}(A)$ cannot approximated from above by co-compact subgroups.
    
    Another interesting example comes from the fact that the free group $\mb F_2$ can be embedded into the the unitary group of the CAR-algebra as a norm-discrete subgroup. Indeed, by the work of Choi \cite{MR540914} the reduced $C^*$-algebra $C^*_r(\PSL(2,\mb Z))$ (and hence $C^*_r(\mb F_2)$) embeds into the Cuntz algebra $\mc O_2$ -- a precursor of Kirchberg's seminal $\mc O_2$-embedding theorem, see \cite{MR1780426}. The Cuntz algebra is a subquotient of the CAR-algebra by a result of Blackadar \cite{MR808296}. Since $\UU(\mc O_2)$ is connected, we can lift the standard unitaries of $C^*_r(\mb F_2)\subset \mc O_2$ into the unitary group of the CAR-algebra, which gives a discrete embedding of $\mb F_2$ into it.
    \end{Rem}
    
    As mentioned already, while the algebra $\mc O_{n}$ is not strongly amenable, it has been observed by Rosenberg \cite{MR467331} that $(\mc O_{n} \otimes \mathcal K)^+$ is strongly amenable; we do not know whether its unitary group is amenable in norm. Even though $(\mc O_{n} \otimes \mathcal K)^+$ is not a limit of noncommutative CW-complexes by results of Cuntz \cite[Section 2.3]{MR467330} and arguments in \cite[page 190]{MR467331}, we think that our strategy to employ the classification program in order to approximate algebras should be extended to cover interesting non-simple algebras. The exact characterization of non-simple algebras that arise as inductive limits of one-dimensional noncommutative CW-complexes is still inconclusive, but important partial results have been obtained. For example, in \cite{MR2578465} Gong et al.\ proved that every ${\rm AH}$-algebra with the ideal property (every ideal is generated by projections) and torsion-free $K$-theory, is an ${\rm AT}$-algebra.
    Inductive limits of  generalized dimension
    drop interval algebras appeared in the study of algebras with real rank zero \cite{MR4400699}. Inductive limits of more general one-dimensional noncommutative CW-complexes featured already in work of Kirchberg and R\o rdam \cite{MR2153904}. Let us also mention the more classical study of ${\rm AI}$-algebras by Thomsen \cite{MR1162672}.
    
    Going into another direction, Elliot's result combined with the recent progress in the classification program shows that all nuclear, simple, separable, stably finite, $\mathcal Z$-stable, UCT algebras are limits of two-dimensional noncommutative CW-complexes. Thus, it seems possible that our main result can be proved without the restriction on $K_0.$
    The basic additional building block here is the algebra $C([0,1]^2,M_n(\mathbb C)).$ It is subject of further work to prove that its unitary group is amenable in the norm topology, which again seems more a problem of stochastic analysis and concentration of measure techniques. However, even if this can be done, there are additional difficulties, since the proof of our main result uses co-compactness at one point which is intrinsic to the one-dimensional situation. It seems unclear at present if this can be rescued so that everything can be reduced to the study of $C([0,1]^2,M_n(\mathbb C)).$
    
    Note that Conjecture \ref{conj1} in particular asserts that for every discrete amenable group $\Gamma$ the unitary group of $C^*(\Gamma)$ is amenable in norm topology. Indeed, in this case $C^*(\Gamma)$ is well-known to be nuclear and easily seen to have the QTS property because an arbitrary quotient of is generated by the image of $\Gamma$ over which one can average. Indeed, every representation of an amenable group is amenable in the sense of Bekka -- and thus admits a trace on the generated $C^*$-algebra even though the generated von Neumann algebra might be infinite. The conjecture seems to be open for the group $\Gamma=\mathbb Z^2 \times {\rm Sym}(3)$.
    
    \subsection{Group of connected components of $\UU(A)$}
    A consequence of Conjecture \ref{conj1} would be that ${\rm U}(A)/{\rm U}_0(A)$ is amenable as a discrete group for every nuclear $C^*$-algebra $A$ with the QTS property. Even though many natural examples of $C^*$-algebras appearing in the classification program satisfy ${\rm U}(A)/{\rm U}_0(A) \cong K_1(A)$ in which case the claim is obvious since $K_1(A)$ is abelian, this consequence is far from being trivial. 
    
    Indeed, it may well happen that ${\rm U}(A)/{\rm U}_0(A)$ is non-abelian. A natural example for this phenomenon is the algebra $A=M_2(C({\rm U}(2)\times {\rm U}(2)))$, using the classical fact that the commutator map $c \colon {\rm U}(2) \times {\rm U}(2) \to {\rm U}(2)$ is not homotopically trivial. Now, if $X$ is a finite cell complex, then it follows from work of Hopkins \cite{MR1017164} that ${\rm U}(M_nC(X))/{\rm U}_0(M_nC(X))$ is a nilpotent group, see also \cite[page 464]{MR516508}. In particular, we see that also in this case, the group is amenable.
    
    \begin{Conj}
    Let $A$ be a nuclear $C^*$-algebra. Then ${\rm U}(A)/{\rm U}_0(A)$ is amenable.
    \end{Conj}
    
    Even though there is a considerable amount of literature on non-stable $K$-theory, to the best of our knowledge, the range of this invariant has not been considered so far.
    
    \subsection{Skew-amenability of $\UU(A)$ in the weak topology} An interesting part of Conjecture \ref{conj1} is the possible implication $(v)\Rightarrow{}(i)$, i.e. skew-amenability of $\UU(A)$ in the weak topology implies that $A$ is nuclear and has the QTS property. This is related to several other notions which we will now briefly discuss.
    
    The key observation here is that every strongly continuous unitary representation $\pi\colon G\to B(H)$ of a skew-amenable group admits a hypertrace, i.e.~a state $\varphi\colon B(H)\to\mb C$ which is $\pi(G)$-invariant; in particular, this gives a trace on the $C^*$-algebra generated by $\pi(G)$ \cite[Theorem 1.3]{PS23}.
    
    This shows that if $\UU(A)$ is skew-amenable, then $A$ is not only QTS, but actually every representation of $A$ has a hypertrace; in other words, $A$ is a hypertracial $C^*$-algebra as defined by B\'edos in \cite{MR1373325}. The results there also show that $(i)$ implies hypertraciality while every hypertracial $C^*$-algebra obviously has the QTS property. B\'edos asked if hypertraciality implies nuclearity; however, counterexamples (simple, AF-embeddable, tracial non-nuclear algebras) were constructed by Dadarlat in \cite{MR1759889} and some more counterexamples were constructed by Brown \cite[Remark 6.2.6]{MR2263412}. 
    
    On the other hand, skew-amenability of $\UU(A)$ in the weak topology actually implies more than hypertraciality. Using the terminology from \cite{MR2263412}, while hypertraciality only means that every quotient has a trace which is amenable, the next result shows that skew-amenability of $\UU(A)$ implies that all such traces are uniformly amenable \cite[Theorem 3.2.2]{MR2263412}.
    
    \begin{Prop}
    Let $A$ be a $C^*$-algebra such that $\UU(A)$ is skew-amenable in the weak topology, $\tau\colon A\to \mb C$ be a trace and $\pi\colon A\to B(L^2(A,\tau))$ the corresponding GNS representation. Then the von Neumann algebra $M = \pi(A)''$ is hyperfinite.
    \end{Prop}
    \begin{proof}
    Let $H\coloneqq \UU(A)$ with the weak topology, $G\coloneqq \UU(M)$ with the strong operator topology; the representation $\pi$ restricts to a continuous homomorphism $\pi\colon H\to G$ with dense image $\pi(H)\leqslant G$. By \cite[Lemma 6.2(2)]{MR4368352}, $\pi(H)$ with the subgroup topology of $G$ is skew-amenable. On the other hand, as $M$ is a finite von Neumann algebra, $G$ is a SIN group. Therefore we conclude that $\pi(H)$, being skew-amenable and SIN, is in fact amenable. Hence, also $G=\UU(M)$ is amenable by Theorem \ref{amen} (\ref{amendense}). However, this implies that $M$ is hyperfinite \cite[Theorem 1]{MR1076577}.
    \end{proof}
    
    It is an interesting open problem whether every skew-amenable group is actually amenable \cite{PS23}; this would in particular give the implication $(v) \Rightarrow (i)$, since $(v)$ implies the QTS property and amenability in the weak topology implies nuclearity by \cite{MR1076577}. Thus, counterexamples to B\'edos' question  might well be examples of $C^*$-algebras with skew-amenable but not amenable unitary group in the weak topology. On the other hand, some of these counterexamples  generate non-hyperfinite von Neumann algebras \cite[Remark 6.2.6]{MR2263412} and therefore do not have skew-amenable unitary groups in view of the above.
    
    \section*{Acknowledgments}
    We thank Martin Schneider, Wilhelm Winter, Hannes Thiel, and Narutaka Ozawa for interesting remarks. We also thank the unknown referee for numerous useful comments. The authors acknowledge funding by the Deutsche Forschungsgemeinschaft (SPP 2026).
    
    Please note also Ozawa's contribution {\it Amenability of unitary groups of simple monotracial $C^*$-algebras} in the same volume, where Conjecture \ref{conj1} is partially confirmed and partially refuted. In particular, it contains first examples of non-nuclear $C^*$-algebras whose unitary groups are skew-amenable (but not amenable!) in the weak topology. 
    
    \bibliographystyle{acm}

    \end{document}